\documentclass[11pt]{amsart}
\usepackage{amsfonts,amssymb,amsmath,amsthm}
\usepackage{anysize}
\usepackage{multirow,array,mathtools,tabularx}
\usepackage[T1]{fontenc}
\usepackage{mathrsfs,bbm,textcomp}
\usepackage[all]{xy}
\usepackage{float}
\usepackage{graphicx,wrapfig,rotating}
\usepackage{caption}
\usepackage{subcaption}
\usepackage[shortlabels]{enumitem}
\usepackage{aliascnt} 
\usepackage{fancyhdr,url,hyperref}
\hypersetup{pdfauthor={Juanita Pinz\'on Caicedo},colorlinks=true,linkcolor=Cyan,citecolor=Cyan}
\usepackage[numbers,sort]{natbib}
\usepackage{setspace}
\usepackage[usenames,dvipsnames]{xcolor}
\usepackage{stackrel}

\marginsize{0.9in}{0.9in}{0.9in}{0.9in}

\newcommand{\Z}{{\ensuremath{\mathbb{Z}}}}

\newcommand{\rst}[1]{\ensuremath{{\big|} \raise-1.25ex\hbox{$\scriptscriptstyle#1$}}}

\makeatletter
\def\MyNewTheorem#1[#2]#3{%
  \newaliascnt{#1}{#2}
  \newtheorem{#1}[#1]{#3}
  \aliascntresetthe{#1}
  \expandafter\newcommand\csname #1autorefname\endcsname{#3}
}
\makeatother

\theoremstyle{MyNewTheorem}
\newtheorem{theorem}{Theorem}
\MyNewTheorem{proposition}[theorem]{Proposition}
\MyNewTheorem{cor}[theorem]{Corollary}
\MyNewTheorem{lemma}[theorem]{Lemma}
\MyNewTheorem{claim}[theorem]{Claim}
\MyNewTheorem{rem}[theorem]{Remark}
\MyNewTheorem{definition}[theorem]{Definition}
\newtheorem*{thm*}{Theorem}

\newtheorem{problem}{Problem}

\makeatletter
\newtheorem*{rep@theorem}{\rep@title}
\newcommand{\newreptheorem}[2]{%
\newenvironment{rep#1}[1]{%
 \def\rep@title{#2 \ref{##1}}%
 \begin{rep@theorem}}%
 {\end{rep@theorem}}}
\makeatother

\newreptheorem{theorem}{Theorem}
\newreptheorem{lemma}{Lemma}

\def\equationautorefname~#1\null{(#1)\null}
\def\itemautorefname~#1\null{#1\null}

\title{Independence of Iterated Whitehead Doubles}
\author{Juanita Pinz\'on-Caicedo}
\address{Department of Mathematics, North Carolina State University,
Raleigh, NC 27695}
\email{jpinzon@ncsu.edu}
\date{}

\begin{document}
\maketitle
\begin{abstract} A theorem of Furuta and Fintushel-Stern provides a criterion for a collection of Seifert fibred homology spheres to be independent in the homology cobordism group of oriented homology 3-spheres. In this article we use these results and some 4-dimensional constructions to produce infinite families of positive torus knots whose iterated Whitehead doubles are independent in the smooth concordance group. 
\end{abstract}
\bigskip

\section{Introduction}

A smooth knot is a smooth embedding of the circle $S^1$ into the 3-sphere $S^3$, and as a consequence of the (unpublished) worked of Thurston, every knot is either hyperbolic, a torus, or a satellite knot \cite{Thurston}. Hyperbolic knots are those whose complement admits a hyperbolic structure, torus knot are those that lie on the surface of an unknotted torus in $S^3$ and are specified by a pair of coprime integers $p$ and $q$, finally, satellite knots are those whose complement contains an incompressible, non boundary-parallel torus. Moreover, satellite knots are constructed from two given knots, $P$ and $K$, in the following way. Let $P\sqcup J$ be a link in $S^3$ with $J$ and unknot so that $P$ lies in the solid torus $V=S^3\setminus N(J)$. The satellite knot with pattern $P\sqcup J$ and companion $K$ is denoted by $P(K)$ and is obtained as the image of $P$ under the embedding of $V$ in $S^3$ that knots $V$ as a tubular neighborhood of $K$, using the 0-framing of $K$.\\

The set of isotopy classes of knots is a semigroup with connect-sum as its binary operation \cite[Chapter 1, Section 5]{murasugi}. 
To obtain a group structure topologists consider another equivalence relation on the set of knots, concordance. Two smooth knots $K_0,K_1\subset S^3$ are smoothly concordant if there exist a smoothly and properly embedded annulus $A$ into the cylinder $I\times S^3$ that restricts to the given knots at each end. However, if the embedding of $A$ into $I\times S^3$ is locally flat and proper, the knots $K_0$ and $K_1$ are said to be topologically concordant. These two different approaches give rise to the two related theories of smooth and topological concordance, and both induce an abelian group structure on the set of knots with connected sum as the operation. Studying the relationship between smooth and topological concordance is an area of active research in knot theory specially because the difference between smooth and topologically slice knots is related to subtle differences in the set of differentiable structures on 4--manifolds. One approach to this problem is to understand the group structure of the ``forgetful homomorphism'' $\mathcal{C}\to \mathcal{C}_{\text{TOP}}.$ The 1980's saw the birth of tremendously important results in that direction: those of Freedman \cite{freedman, freedman2, freedman-quinn}  and Donaldson \cite{d1,d2,d3}. These results allowed topologists to actually understand just how vastly different $\mathcal{C}$ and $\mathcal{C}_{\text{TOP}}$ are. Indeed, on the one hand, Freedman's theorem implies that knots with Alexander polynomial 1 are topologically slice, and on the other, Donaldson's theorem can be used to show that some knots with Alexander polynomial 1 are not smoothly slice. As an example, if the link $D^r\sqcup J$ is the $r$-th iterated Whitehead link, $D^r(K)$ is called the (untwisted positively clasped) $r$-th iterated Whitehead double of $K$. Since the pattern $D^r$ is an unknot in $S^3$ and is trivial in $H_1(V; \Z)$, then the satellite $D^r(K)$ has trivial Alexander polynomial and is thus topologically slice. Part of the motivation to study independence of Whitehead doubles comes from the following two problems in Kirby's list \cite{kirby-problems}. 

\begin{problem}[1.38 in \cite{kirby-problems}] The untwisted double of a knot is slice if and only if the knot is slice.\end{problem}

\begin{problem}[1.39 in \cite{kirby-problems}] Is any $D^r\sqcup J$ null-concordant?\end{problem}

It was Akbulut \cite{akbulut} who first used Donaldson's theorem to prove that $D(T_{2,3})$, the Whitehead double of the right handed trefoil knot, is not smoothly slice. Akbulut's technique was extended by Cochran-Gompf \cite{cochran-gompf} to show that Whitehead doubles of positive torus knots have infinite order. The introduction of Seiberg-Witten invariants and their relationship with the invariant of Thurston-Benequin allowed Rudolph \cite{r1,r2} to show that iterated doubles of strongly quasipositive knots (of which torus knots are an example) are not smoothly slice. Next, the work of Ozsv\'ath and Szabo in Heegaard Floer was used by Hedden \cite{hedden} to generalize the result of Rudolph, and by Park \cite{park} to show that $D(T_{p,2m+1})$ and $D^2(T_{p,2m+1})$ generate a rank-2 subgroup of $\mathcal{C}$. Finally, the work of Furuta \cite{furuta} and Fintushel-Stern \cite{FS-inst}  on the theory of instantons and Chern-Simons invariants was used by Hedden-Kirk \cite{hedden-kirk} to show that certain families of Whitehead doubles of positive torus knots generate a subgroup of $\mathcal{C}$ of maximal infinite rank. Their work was generalized to other Whitehead-like patterns in \cite{tesis} and the present work can be regarded as its sequel.\\

As usual with concordance, the methods involve the topology of 3-- and 4--manifolds. If a knot $K$ is slice, then on one hand, the 2-fold cover of $S^3$ branched over $K$ bounds a smooth 4--manifold with the same $\Z/2$ homology as the 4--ball, and on the other, surgery on $S^3$ along $K$ bounds a 4--manifold with an embedded 2-sphere. Then, obstructions to these 3--manifolds from bounding smooth 4--manifolds with the prescribed topology provide obstructions to the knots used in their constructions from being slice. In the present article the obstruction to the sliceness of sums of Iterated Doubles will come from the study of the moduli space of ASD connections on a 4--manifold with cylindrical ends modeled over the 2-fold covers of $S^3$ branched over the doubles.\\

\textbf{Acknowledgements:} I would like to thank the organizers of the conference ``Topology in dimension 3.5'' since the conference provided a favorable environment to conceive the first ideas of the present work. I would also like to thank the organizers and participants of the ``Conference on 4-manifolds and knot concordance'' for bearing with me as I presented an earlier and mistaken version of the main result of the present article. Special thanks go to David Krcatovich, Marco Golla, and Marco Marengon for spending enough time paying careful attention to my arguments and finally finding the mistake in my logic, and to Matt Hedden for listening to the correct proof of the theorem.\\

\textbf{Outline:} To establish the result we will translate the question of concordance of knots into a question of homology cobordism of their 2-fold covers, and so \autoref{sec::covers} focuses on a description of the topology of the covers of iterated Whitehead doubles. The basics of the results of Furuta \cite{furuta} and Fintushel-Stern \cite{FS-inst} regarding independence in $\Theta^3_{\Z/2}$ will be briefly presented in \autoref{sec::gauge}. However, this result requires a particular geometric structure on 3--manifolds which 2-fold covers over iterated doubles lack. This issue can be sidestepped via cobordisms and so \autoref{sec::cobordisms} describes constructions of definite cobordisms from our 3--manifolds to 3--manifolds with the right geometry. With all the ingredients at hand, \autoref {sec::proof} presents the final proof.

\section{Satellites and their cover}\label{sec::covers}

The classification of knots up to concordance, and in particular the study of sliceness, is a problem that involves both 3-- and 4--manifold topology. Double branched covers, for example, provide a connection between problems in knot theory and other questions in low-dimensional topology. To be more specific, obstructions to sliceness can be found using homology cobordism, a 3--dimensional analogue to concordance. Call two oriented $\Z/2$--homology spheres $\Sigma_0$ and $\Sigma_1$ homology cobordant if there is an oriented smooth 4--manifold $W$ with oriented boundary $-\Sigma_0 \sqcup \Sigma_1$ and such that $H_*(W;R)=H_*(I\times S^3;R)$. In this case, we call $W$ a cobordism from $\Sigma_0$ to $\Sigma_1$, and we call it a negative or positive definite cobordism if its intersection form is respectively negative definite or positive definite. The set of homology cobordism classes of $\Z/2$--homology spheres forms an abelian group denoted by $\Theta_{\Z/2}^3$ and with connected sum as the group operation. To establish the relationship between homology cobordism and concordance, for a knot $K$ in $S^3$, denote 
by $\Sigma(K)$ the 2-fold cover of $S^3$ branched over $K$. Notice that if two knots $K_0$ and $K_1$ are concordant via an annulus $A$, then the double cover of the cylinder $I\times S^3$ branched along $A$ is a $\Z/2$-homology cobordism between the 3--manifolds $\Sigma(K_0)$ and $\Sigma(K_1)$ and so there is a well defined assignment $K\to\Sigma(K)$ from the smooth concordance group $\mathcal{C}$ into $\Theta^3_{\Z/2}$. Moreover, since the separating $S^2$ that appears in a connected sum $K_0\# K_1$ lifts to a separating sphere of the 2-fold cover of the sum, this map is a group homomorphism. Thus, the question of the existence (or lack thereof) of a slicing disk for a knot $K$ translates into the question of the existence of a 4--manifold with the same $\Z/2$-homology groups as the 4--ball and with boundary $\Sigma(K)$. This section details a special and particularly useful decomposition of branched covers of satellites.\\

First, recall that the untwisted satellite knot $P(K)$ with pattern $P$ and companion $K$ is obtained as the image of $P$ under the embedding of an unknotted solid torus $V$ in $S^3$ containing $P$ that knots $V$ as a tubular neighborhood of $K$, using the 0-framing of $K$. For $n$ any integer, an $n$-twisted satellite is obtained using the $n$-framing of $K$ in the identification. Whitehead doubles are an important example of untwisted satellites and are obtained by using the Whitehead link as the pattern of the satellite operation. Similar examples arise by considering Whitehead doubles of other Whitehead doubles. These examples are called iterated Whitehead doubles. The images to the right of \autoref{fig::iterated} show the patterns of the first three iterations of this satellite operation. \\

Now, regarding general covers of $S^3$ branched over satellites, denote by $\Sigma_q(P)$ the $q$-fold branched cover of $S^3$ branched over $P(K)$. If $J$ is a parallel copy of a meridional curve for the solid torus $V$ so that $V$ can be identified with $S^3\setminus N(J)$, then $J$ has $l$ lifts into $\Sigma_q(P)$, where $l=\text{gcd}(q, lk(J,P))$. It follows that $\Sigma_q(P(K))$, the $q$-fold branched cover of $S^3$ branched over $P(K)$, is formed from $\Sigma_q(P)$ by removing neighborhoods of the lifts of $J$ and replacing each with the $q/l$-cyclic cover of the complement of $K$. See \cite{Livingston-Melvin} and \cite{Seifert} for the details. For Whitehead doubles and $q=2$ specifically, we have the following decomposition. The details can be found in \cite{hedden-kirk} or \cite{tesis}.
\begin{equation}\label{r=1}\Sigma\left(D(K)\right)=S^3\setminus N\left(T_{2,4}\right)\underset{\widetilde{\phi}}{\cup} 2 E(K).\end{equation} 
Here $E(K)$ denotes the knot exterior $S^3\setminus N(K)$ and the gluing map $\widetilde{\phi}$ identifies  each copy of the meridian $\mu_K$ of the knot $K$ to the curve $(-2,1)$ in each component of $\partial N\left(T_{2,4}\right)$. See \autoref{fig::lifts} for a visualization of this cover. This description is obtained by identifying the covering space of $S^3\setminus N(D)$ with $S^3\setminus N(T_{2,4})$, where $T_{2,4}$ is the (unoriented) two component torus link determined by $(2,4)$. The following proposition relies on this identification to describe a similar decomposition for the double covers $\Sigma\left(D^r(K)\right)$ for $r>1$.

\begin{figure}
\centering
\def\svgwidth{0.25\textwidth}
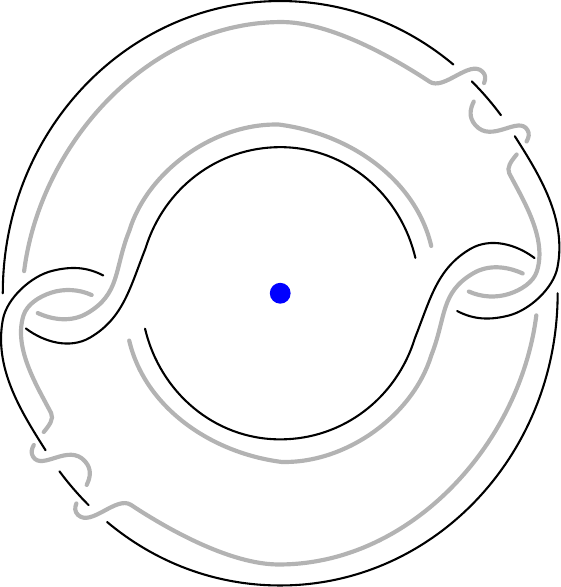$\qquad\to\qquad$
\def\svgwidth{0.25\textwidth}
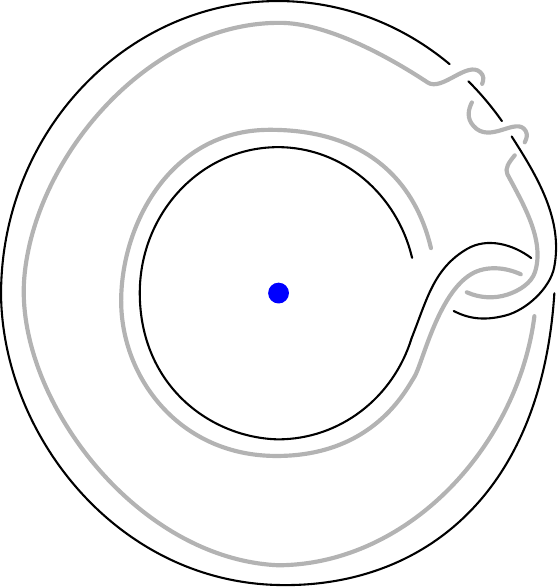
\caption{The cover $\Sigma\left(D^1\right)\to S^3$, the knot $J$ and its longitude, as well as their lifts.}\label{fig::lifts}
\end{figure}

\begin{proposition}\label{covers} Given a knot $K\subseteq S^3$ and an integer $r\geq 1$, let $D^r(K)$ be the $r$-th iterated Whitehead double of $K$. The 2-fold cover $\Sigma\left(D^r(K)\right)$ of $S^3$ branched over $D^r(K)$ has a decomposition $$\Sigma\left(D^r(K)\right)=S^3\setminus N\left( D_{-2}^{r-1}(T_{2,4})\right)\underset{\widetilde{\phi}}{\cup} 2 E(K),$$ where $2E(K)$ denotes two disjoint copies of the knot exterior $S^3\setminus N(K)$, and $D_{-2}^{r-1}(T_{2,4})$  the iterated double of the link $T_{2,4}$ twisted by $-2$. Additionally, the gluing map $\widetilde{\phi}$ identifies each copy of the meridian $\mu_K$ of the knot $K$ with the longitude of one of the components of  of $D_{-2}^{r-1}(T_{2,4})$.
\end{proposition}

\begin{proof}
Denote by $\Sigma(V,D^r)$ the 2-fold cover of the solid torus $V=S^3\setminus N(J)$ branched over the pattern $D^r$. Results from \cite{Livingston-Melvin} and \cite{Seifert} show that the 2-fold cover of $S^3$ branched over $D^r(K)$ is completely determined by $\Sigma(V,D^r)$, the homology class of the two lifts of $\mu_V=\lambda_J$ to $\Sigma(V,D^r)$, and the knot exterior $S^3 \setminus N(K)$. Since the iterated Whitehead link $D^r\sqcup J$ is symmetric, $\Sigma(V,D^r)$ is diffeomorphic to the 2-fold cover of $S^3\setminus N(D^r)$ branched over $J$. The latter space can be obtained from $\Sigma(S^3,J)\cong S^3$ after removing a tubular neighborhood of the lift of $D^r$, and so a description of the lift of $D^r$ and the lift of its longitude will be enough to completely understand $\Sigma\left(D^r(K)\right)$. With that in mind, consider the untwisted satellite map $\psi: V\to N(D)$ with pattern $D^{r-1}$ and realize $D^r\subseteq N(D)$ as $D^r=D^{r-1}(D)=\psi(D)$. Next, related to the identification included as \autoref{r=1}, notice that if $p:\Sigma(S^3,J)\to S^3$ is the cover map, then the restriction of $p$ to $N(T_{2,4})=p^{-1}(N(D))$ is a cyclic covering space $p:N(T_{2,4})\to N(D)$ and so basic covering space theory \cite[Proposition 1.33]{Hatcher} shows that corresponding to the map $\psi:V\to N(D)$, and for each choice of component $A_i$ of $T_{2,4}$, there exists a homeomorphism $\widetilde{\psi}_i: V\to N(A_i)$ satisfying $p\circ \widetilde{\psi}_i=\psi$. Moreover, $\widetilde{\psi}_i$ carries $\mu_{V}$ to $\mu_{A_i}$, and $\lambda_{V}$ to $-2\mu_{A_i}+\lambda_{A_i}$ and therefore the lift of $D^r$ to $N\left(T_{2,4}\right)$ is given by $\widetilde{\psi}_1(D^{r-1})\sqcup \widetilde{\psi}_2(D^{r-1})=D_{-2}^{r-1}(T_{2,4})$, the $-2$-twisted $r-1$ iterated double of the link $T_{2,4}$. Finally, since each $\widetilde{\psi}_i$ is a twisted satellite map in its own right, they identify the longitude of $D^{r-1}$ with a longitude of its image $D_{-2}^{r-1}(A_i)$ thus showing that the lifts of a longitude of $D^r$ to $N(T_{2,4})\subset \Sigma(S^3,J)$ are precisely the longitudes of the components of $D_{-2}^{r-1}(T_{2,4})$.
\end{proof}

\begin{figure}[h]
\centering
\def\svgwidth{0.25\columnwidth} \input{Cover_1.pdf_tex}
\def\svgwidth{0.25\columnwidth} \input{Base_1.pdf_tex}\\

\def\svgwidth{0.25\columnwidth} 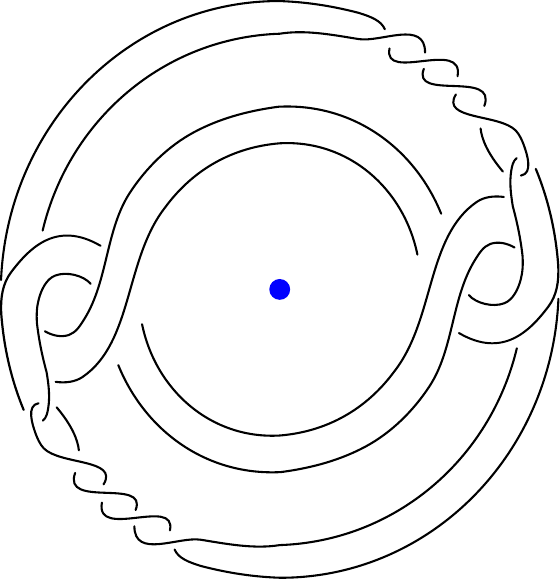
\def\svgwidth{0.25\columnwidth} 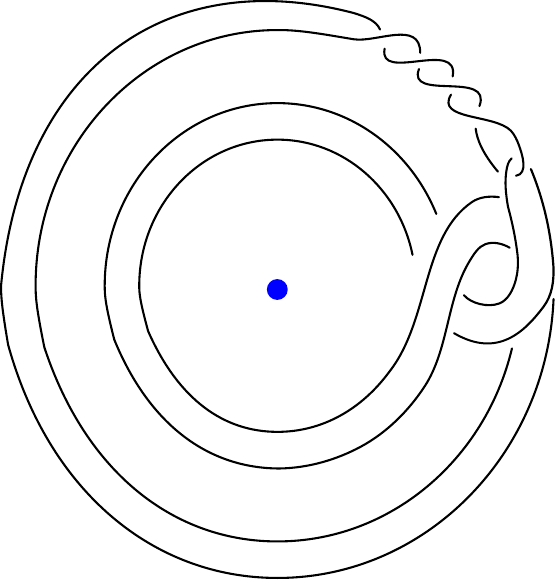\\

\def\svgwidth{0.25\columnwidth} 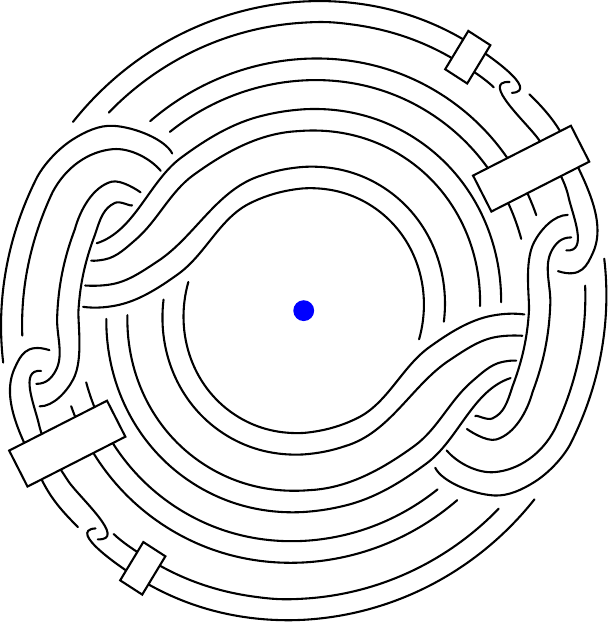
\def\svgwidth{0.25\columnwidth} 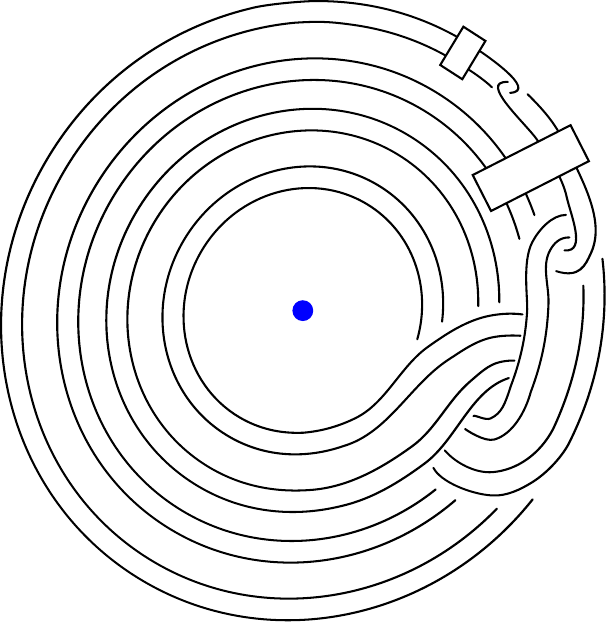
\caption{The cover $\Sigma\left(D^r\right))\to S^3$ for $r=1,2,3$.}\label{fig::iterated}
\end{figure}

\section{The 4--dimensional obstruction}\label{sec::gauge}
Since for every $r\geq 1$ the knot $D^r$ is trivial in $S^3$, Freedman's theorem \cite{freedman, freedman2, freedman-quinn} implies that every iterated untwisted Whitehead double is topologically slice. As a consequence, classical invariants do not detect information about their smooth concordance type and so  smooth techniques like gauge theory are necessary to obtain that information. In this article we will use the internal structure of the moduli space of anti-self dual connections on a Seifert fibration with three exceptional fibers to obtain an obstruction to the sliceness of sums of iterated Whitehead doubles of torus knots. Similar to the way in which Donaldson's theorem applies instantons to describe the specific form the intersection form of a closed and definite 4--manifold takes, the theory of instantons can be applied to manifolds with boundary (via the addition of cylindrical ends) to obstruct the existence of certain 4--manifolds. This theory was first developed by Furuta \cite{furuta} and Fintushel and Stern \cite{FS-inst}, and extended by Hedden and Kirk \cite{hedden-kirk}. For the technical details we refer the reader to their articles. The following theorem is a restatement of their combined results that is free of the technicalities of gauge theory. To establish notation, given $(p,q,s)$ a set of relatively prime and positive integers, let $\Sigma(p,q,s)$ be the Seifert fibered homology sphere with exceptional fibers of orders $p,q,s$.

\begin{theorem}\label{cobound}Let $p_i,q_i$ be relatively prime integers and $k_i$ a positive integer for $i=1,\ldots,N$. If $\left\{\Sigma_i\right\}_{i=1}^N$ is a family of Seifert fibred homology 3-spheres such that $\Sigma_i=\Sigma(p_i,q_i,k_ip_iq_i-1)$ and satisfying
\begin{equation}\label{criterion}
p_iq_i(k_ip_iq_i-1)<p_{i+1}q_{i+1}(k_{i+1}p_{i+1}q_{i+1}-1),\end{equation}
then no linear combination of elements in $\left\{\Sigma_i\right\}_{i=1}^N$ cobounds a smooth 4-manifold $X$ with negative definite intersection form and such that $H_1(X;\Z/2)=0$.
\end{theorem}

The result is obtain after examining the moduli space $\mathcal{M}$ of ASD connections on an $SO(3)$ bundle over a 4--manifold $X$ with boundary a sum of elements of $\left\{\Sigma_i\right\}_{i=1}^N$, and determined by fibration $\Sigma(p_N,q_N,k_Np_Nq_N-1)\to S^2$. The (virtual) dimension of $\mathcal{M}$ can be calculated using the Neumann-Zagier formula \cite{nz} and can be shown to be exactly 1 when $p_N$ and $q_N$ are relatively prime positive integers, and $k_N\geq1$. Next, since $\mathcal{M}$ is a 1--dimensional space, its singularities can be regarded as boundary points. In addition, the singularities of $\mathcal{M}$ can be shown to correspond to reducible connections, the number of which can be computed in terms of the order of the torsion subgroup of $H_1(X;\Z)$ and the number of even factors in it. Finally, compactness of $\mathcal{M}$ is guaranteed after requiring $1/4<p_{i+1}q_{i+1}(k_{i+1}p_{i+1}q_{i+1}-1)$ to rule out bubbling, and $p_iq_i(k_ip_iq_i-1)<p_{i+1}q_{i+1}(k_{i+1}p_{i+1}q_{i+1}-1)$ to rule out leaking. If the 4--manifolds $X$ was negative definite and had $H_1(X;\Z/2)=0$, then $\mathcal{M}$ would be a compact 1--dimensional space with an odd number of boundary components, which is a contradiction and so at least one of the hypothesis about the topology of $X$ had to be incorrect.\\

This purely 4--dimensional theorem can be used to obtain an obstruction to sliceness as follows: suppose a certain sum of knots is slice, then the sum of the double covers bounds a smooth 4--manifold $Q$ with $\Z/2$ homology isomorphic to the $\Z/2$ homology of $B^4$. However, if the double covers are cobordant to Seifert fibered spheres via cobordisms satisfying the hypothesis of \autoref{cobound}, then the manifold $X$ obtained as the union of $Q$ and the cobordisms would be a negative definite manifold with $b_1=0$.  The existence of $X$ contradicts \autoref{cobound} and so the original hypothesis of the sliceness of the sums was incorrect. The next section describes some special constructions of cobordisms from double covers of iterated doubles to Seifert fibered spaces.

\section{Cobordisms}\label{sec::cobordisms}
Since the covers $\Sigma(D^r(K))$ are not Seifert fibered spaces, the results explained in \autoref{sec::gauge} cannot be applied directly to their sums. However, the 4--dimensional obstruction allows for $b_2>0$ as long as the 4--manifold is positive definite. In this section we will take advantage of this fact and construct definite cobordisms from the covers to some Seifert fibered spaces. It is worth mentioning that positive definite cobordisms are necessary if we want to consider sums of iterated Whitehead doubles that include negative coefficients. In fact, Donaldson's theorem and its treatment in \cite{cochran-gompf} show that sums of iterated doubles with only positive coefficients are never slice. \\

\begin{lemma}\label{cover-to-splice} Let $r>1$. If $K$ admits a series of positive-to-negative crossing changes that transform it into an unknot, then there exists a negative definite cobordism from $\Sigma\left(D^r(K)\right)$ to $Splice(D^{r-1}_{-2}(U),K)$.\end{lemma}

\begin{proof} 
By hypothesis $K$ admits a series of positive-to-negative crossing changes that transform it into an unknot. That is, there exists a sequence of positive-to-negative crossing changes such that the $i$-th crossing change is obtained by performing $-1$ surgery on $S^3$ along a trivial knot $\gamma_i$ that lies in $E(K)$, encloses the crossing, and has linking number $0$ with the knot $K$. Next, consider the description of the double cover described in \autoref{covers}, and notice that $\gamma_i$ is contained in $E(K)$, and that the gluing map identifies $\mu_K$ with the longitude of one of the components of $D^{r-1}_{-2}(T_{2,4})$. Thus, $\gamma_i$ can be regarded as a subset of $\Sigma=\Sigma(D^r(K))$ with framing precisely given by its framing in $S^3$. Form a 4--manifold $Z$ by attaching 2--handles to $I\times \Sigma$ along the framed circles $\gamma_i$.\\ 

First, notice that the oriented boundary of $Z$ consists of the disjoint union of $-\Sigma$ and the result of surgery on $\Sigma$ along the knots $\gamma_i$ with framing number $-1$. Denote by $Y$ the latter manifold and notice that as a consequence of \autoref{covers}, $Y$ can be seen to split as the union of $ E(K) \underset{\varphi_1}{\cup} \left(S^3\setminus N(D^r_{-2}(T_{2,4}))\right)$ and the result of $-1$--surgery on $E(K)$ along the $\gamma_i$'s. Since the $\gamma_i$'s were chosen to give an unknotting sequence for $K$, $-1$--surgery on $E(K)$ along the $\gamma_i$'s is isomorphic to the unknot complement and therefore isomorphic to a standard solid torus $D^2\times S^1$. Furthermore, the isomorphism preserves meridian-longitudes pairs and thus the Seifert longitude of $K$ gets sent to a meridional curve $\partial D^2\times\{\text{pt.}\}$ of $D^2\times S^1$ and the meridian of $K$ gets sent to the longitudinal curve $\{\text{pt.}\}\times S^1$ of the solid torus $D^2\times S^1$. Then, there is an isomorphism $$Y\cong\left(S^3\setminus N(K)\right)\underset{\varphi_1}{\cup}\left(S^3\setminus N(D^r_{-2}(T_{2,4}))\right) \underset{\varphi'_2}{\cup} D^2\times S^1,$$ where $\varphi'_2$ identifies $\partial D^2\times\{\text{pt.}\}$ with a longitude of one of the components of $D^r_{-2}(T_{2,4})$. Specifically, if $A_1,A_2$ are the components of the link $T_{2,4}$ and $\mu_2,\;\lambda_2$ are respectively, a meridian and longitude of $D^r_{-2}(A_2)$, then the gluing map $\varphi'_2$ satisfies $$\left(\varphi'_2\right)_*(\left[S^1\right])=(\varphi_2)_*(\mu_K)=\lambda_{2}\quad\text{ and }\quad \left(\varphi'_2\right)_*(\left[\partial D^2\right])=(\varphi_2)_*(\lambda_K)=\mu_{2}.$$ Therefore $\varphi'_2$ extends to the interior of $D^2\times S^1$ and we have $$Y\cong\left(S^3\setminus N(K)\right)\underset{\varphi_1}{\cup}\left(S^3\setminus N(D^{r-1}_{-2}(A_1))\right),$$ with $$(\varphi_1)_*(\mu_K)=\lambda_{1}\quad\text{ and }\quad (\varphi_1)_*(\lambda_K)=\mu_{1}.$$
This shows that $Y$ is diffeomorphic to $Splice(K,D^{r-1}_{-2}(U))$.\\

Next, to see that $Z$ is negative definite, notice that since $\Sigma$ is a homology sphere, the second homology group $H_2(Z;\Z)$ admits a basis determined by the 2--handles, and the matrix representation of the intersection form of $Z$ in terms of this basis is given by the linking matrix of the $\gamma_i$'s. It should be clear that if $c$ is the number of crossing changes in the unknotting sequence for $K$, this matrix is $-I_c$, where $I_c$ is the $c\times c$ identity matrix. 
\end{proof}

\begin{lemma}\label{splice-to-surgery} Let $Splice(K_0,K_1)$ denote the splice of two knots $K_0,K_1\subseteq S^3$. There exists a negative definite cobordism from $Splice(K_0,K_1)$ to $S^3_{+1}\left(K_0\right)\#S^3_{+1}\left(K_1\right)$.
\end{lemma}

\begin{proof}
First, consider surgery descriptions for $K_0$ and $K_1$ consisting of a link $L_i\subset S^3$ ($i=0,1$) whose first component represents $K_i$ and the remaining components are unknotted circles in $S^3$ with zero linking number with $K_i$ and with framing $\pm 1$. Then $Splice(K_0,K_1)$ has surgery diagram given by linking the first component of $L_0$ with the first component of $L_1$ in a way reminiscent of the linking of the components of the positive Hopf link. See \cite[Figure 1.4, pg. 9]{saveliev}.\\
Next, unlink the sub links $L_0$ and $L_1$ to produce the cobordism. Specifically, consider $\gamma$ an unknotted curve in $S^3$ that links the first component of each $L_0$ and $L_1$ exactly once. The formula found in \cite[Lemma 1.2]{hoste} shows that if $\gamma'$ is a curve in the boundary of a tubular neighborhood of $\gamma$ that represents the homology class $(m,1)$, then the linking number of $\gamma$ and $\gamma'$ in $Splice(K_0,K_1)$ is given by $$lk(\gamma,\gamma';Splice(K_0,K_1))=lk(\gamma,\gamma';S^3)-2=m-2.$$ This shows that the 4--manifold obtained by attaching a 2-handle to $I\times Splice(K_0,K_1)$ along $\gamma$ with framing $m$ is negative definite as long as $m\leq 1$. Thus, if we choose $m=-1$, Kirby calculus shows that $-1$ surgery on $Splice(K_0,K_1)$ along $\gamma$ unlinks $L_0$ and $L_1$ and changes the farming of the first component of each $L_0$ and $L_1$ from 0 to 1. It is easy to see that $L_i$ with this new framing is a surgery description of $S^3_{+1}(K_i)$. In addition, since $L_0$ and $L_1$ are now unlinked, there is a separating two-sphere and thus $-1$ surgery on $Splice(K_0,K_1)$ along $\gamma$ is diffeomorphic to $S^3_{+1}\left(K_0\right)\#S^3_{+1}\left(K_1\right)$. \end{proof}

\begin{theorem}\label{cobordisms}

Let $K$ be any knot and $\Sigma\left(D^r(K)\right)$ the 2--fold cover of $S^3$ branched over $D^r(K)$. For $p,q$ a pair of integers that are relatively prime and positive, denote by $T_{p,q}$ the torus knot determined by $p$ and $q$. There exist 4--manifolds $Z_{p,q}$, $P_{p,q}$ and $R_{p,q}$ such that
\begin{enumerate}[label=(\alph*),ref=Theorem \thetheorem (\alph*)]
\item\label{top_cob} $Z_{p,q}$ is a negative definite cobordism from $\Sigma(D^r(T_{p,q}))$ to $-\Sigma(p,q,pq-1)$,
\item\label{pos} $P_{p,q}$ is a positive definite cobordism from $\Sigma(D^r(T_{p,q}))$ to $-\Sigma(p,q,4pq-1)\#-\Sigma(p,q,4pq-1)$, and
\item\label{neg} $R_{p,q}$ is a negative definite manifold with oriented boundary $-\Sigma(D^r(T_{p,q}))$.\end{enumerate}
\end{theorem}

\begin{proof} To start, notice that \autoref{neg} follows immediately from \cite[Theorem 5.3(b)]{tesis}. Similarly, the case $r=1$ in \autoref{top_cob} follows from \cite[Theorem 5.3(a)]{tesis} and the fact that $-1$-surgery on $S^3_{1/2}\left(T_{p,q}\right)$ along a longitude of $T_{p,q}$ gives $S^3_{+1}\left(T_{p,q}\right)$ through a negative definite cobordism as in \cite[Lemma 2.11]{cochran-gompf}.\\

Next, to obtain  $Z_{p,q}$ as in \autoref{top_cob}, notice that for $r>1$ \autoref{cover-to-splice} and \autoref{splice-to-surgery} give a negative definite cobordism $W$ from $\Sigma(D^r(T_{p,q}))$ to $S^3_{+1}\left(D^{r-1}_{-2}(U)\right)\#S^3_{+1}\left(T_{p,q}\right)$. Consider a curve $\gamma$ in $S^3$ enclosing the crossings in the smallest clasp of $D^{r-1}_{-2}(U)$ and with linking number zero with the knot $D^{r-1}_{-2}(U)$. Then $-1$ surgery on $S^3$ along $\gamma$ unknots $D^{r-1}_{-2}(U)$. Next, since $lk(\gamma,D^{r-1}_{-2}(U);S^3)=0$, \cite[Lemma 1.2]{hoste} shows that the framing number of $\gamma$ as a knot in $S^3_{+1}\left(D^{r-1}_{-2}(U)\right)$ is also $-1$ and so the manifold $Z_{p,q}$ obtained from attaching a 2--handle to $W$ along $\gamma$ is a negative definite cobordism from $\Sigma(D^r(T_{p,q}))$ to $S^3_{+1}\left(T_{p,q}\right)$. A theorem of Moser \cite{moser} shows that $S^3_{+1}\left(T_{p,q}\right)=-\Sigma(p,q,pq-1)$.\\
 
Lastly, to obtain the cobordism $P_{p,q}$ from \autoref{pos}, consider $+1$-framed unknotted curves $\delta_1,\delta_2$ in $S^3$ that enclose the smallest clasp in each of the components of the link $D^{r-1}_2(T_{2,4})$ with linking number $2$ and framing $+1$. Let $D_0$ be the restriction to $S^3\setminus N(D^{r-1}_2(T_{2,4}))$ of a spanning disk for $\delta_i$. Then, since $\varphi_i$ identifies the meridian of each component of $D^{r-1}_2(T_{2,4})$ with a longitude of $T_{p,q}$, the union of $D_0$ with two copies of a Seifert surface for $T_{p,q}$ in $E(T_{p,q})_i$ is a Seifert surface for $\delta_i$ in $\Sigma(D^r(T_{p,q}))$ and so the framing number for $\delta_i$ in $\Sigma(D^r(T_{p,q}))$ equals its framing number in $S^3$. This shows that the 4--manifold obtained by attaching 2-handles to $I\times \Sigma(D^r(T_{p,q}))$ along $\delta_1,\delta_2$ in $\{1\}\times \Sigma(D^r(T_{p,q}))$ is a positive definite cobordism from $\Sigma(D^r(T_{p,q}))$ to $(S^3\setminus N(U_2))\cup 2(S^3\setminus N(T_{p,q}))$, where $U_2$ is the 2-component trivial link with framing $-4$. This last 3--manifold can be seen to be $S^3_{1/4}(T_{p,q})\#S^3_{1/4}(T_{p,q})$ and the same theorem of Moser mentioned before (\cite{moser}) shows that $S^3_{1/4}(T_{p,q})=-\Sigma(p,q,4pq-1)$.
\end{proof}

Notice that \autoref{top_cob} also exists for any knot $K$ that admits a series of positive-to-negative crossing changes that transform it into a torus knot $T_{p,q}$ with $p,q>0$, and that \autoref{neg} exists for any knot $K$.

\begin{lemma}[Lemma 2.10 from \cite{cochran-gompf}]\label{pos-to-neg} If $J$ admits a series of positive-to-negative crossing changes that transform it into $K$, then for any rational number $q$ there exists a negative definite cobordism from $S^3_q(J)$ to $S^3_q(K)$.
\end{lemma}

\begin{proof} 

Since $J$ admits a series of positive-to-negative crossing changes that transform it into $K$, there exists a link $L\subseteq S^3\setminus N(J)$ with linking number 0 with $J$ and such that each component of $L$ encloses one of the crossings of $J$ that will be changed. Also, each component of $L$ has framing $-1$. Then, the 4-manifold $W$ obtained by attaching 2-handles to $I\times S^3$ along the framed link $L$ is negative definite, and the annulus $A$ equal to the inclusion $I\times J\to I\times S^3$ regarded as a submanifold of $W$ restricts to $J$ and $K$ at each end of $W$ and has trivial normal bundle in $W$. Denote by $W_q(A)$ the manifold obtained by removing $A \times D^2$ from $W$ and replacing it with $I \times S^1 \times D^2$ using a homeomorphism which is a product on the $I$ factor and restricts on $\partial I$ to the $q$ surgeries on $\partial W$. Then $H_2(W_q(A))\cong H_2(W\setminus A) \cong H_2(W)$ and hence both 4--manifolds have equivalent intersection forms. Thus $W_q(A)$ is a negative definite cobordism from $S^3_q(J)$ to $S^3_q(K)$.\\
\end{proof}

\section{Independence}\label{sec::proof}
This section is the culmination of all the work. Here we mix all the ingredients that we developed in the previous sections to finally obtain the main result, as a corollary to the following 3--dimensional theorem.

\begin{theorem} Let $\left\{\left(p_i,q_i\right)\right\}_i$ be a sequence of relatively prime positive integers and $r_i$ a positive integer for every $i$. If $$p_{i}q_{i}\left(4p_{i}q_{i}-1\right)<p_{i+1}q_{i+1}\left(p_{i+1}q_{i+1}-1\right),$$ then the family $\mathscr{F}=\left\{\Sigma\left(D^{r_i}\left(T_{p_i,q_i}\right)\right)\right\}_{i=1}^\infty$ is independent in $\Theta^3_{\Z/2}$.
\end{theorem} 

\begin{proof} Denote by $[Y]$ the homology cobordism class of the $\Z/2$--homology sphere $Y$ and suppose by contradiction that there exist integral coefficients $c_1,\ldots,c_N\in\Z$ such that $$\sum_{i=1}^Nc_i\left[\Sigma\left(D^{r_i}\left(K_i\right)\right)\right]=0$$ in $\Theta^3_{\Z/2}$. The supposition implies the existence of an oriented 4-manifold $Q$ with the $\Z/2$ homology of a punctured $4$--ball and with boundary
$$\partial Q=\overset{N}{\underset{i=1}{\#}}\left(\overset{c_i}{\underset{j=1}{\#}}\Sigma\left(D^{r_i}\left(K_i\right)\right)\right).$$
Attaching 3--handles to $Q$ we can further assume that $$\partial Q=\bigsqcup_{i=1}^Nc_i\Sigma\left(D^{r_i}\left(K_i\right)\right).$$
Here we use $cY$ to denote the disjoint union of $c$ copies of $Y$ if $c>0$, and $-c$ copies of $-Y$ if $c<0$. In addition, and without loss of generality, further assume that $c_N\geq 1$. Augment $Q$ using the cobordisms constructed in \autoref{cobordisms}, namely, let $$X=Q\cup\bigg(Z_{p_N,q_N}\bigg)\cup\left(\bigsqcup_{c_i>0} R_{p_i,q_i}\right)\cup\left(\bigsqcup_{c_i<0} -P_{p_i,q_i}\right).$$
Thus, $X$ is a negative definite 4--manifold with oriented boundary $$\partial X=-\Sigma(p_N,q_N,p_Nq_N-1)\sqcup \left(\bigsqcup_{c_i<0} \Sigma(p_i,q_i,4p_iq_i-1)\right).$$
Additionally, since the first $\Z/2$--homology groups of $Z_{p_N,q_N}$, $-P_{p_i,q_i}$, $R_{p_i,q_i}$, and $Q$ are trivial, the Mayer-Vietoris theorem shows that $H_1(X,\Z/2)=0$. This would imply that the Seifert fibered spaces $-\Sigma(p_N,q_N,p_Nq_N-1)\sqcup \left(\bigsqcup_{c_i<0} \Sigma(p_i,q_i,4p_iq_i-1)\right)$ cobound a smooth 4--manifold that has negative definite intersection form and that satisfies $H_1(X,\Z/2)=0$, contradicting \autoref{cobound}. Therefore, $Q$ cannot exist and so the 3-manifolds $\Sigma\left(D^{r_i}\left(K_{i}\right)\right)$ are independent in the $\Z/2$ homology cobordism group.
\end{proof}

To summarize, to show that a sum $\#_{i=1}^Nc_iY_i$ of $\Z/2$ of homology spheres does not bound a putative $\Z/2$ homology ball $Q$, attach negative definite cobordisms to either cap off $Q$, or to simplify its boundary. Then use gauge theoretical techniques to rule out the existence of the 4--manifold obtained, and thus of the putative ball itself. Ideally, one would obtain a closed 4--manifold and then apply Donaldson's theorem since that would give the strongest possible result. In the case under consideration, sums with only positive coefficients pose no problems since double covers of Whitehead doubles bound negative definite manifolds and so can be capped-off. However, sums that involve positive coefficients pose a problem. Indeed, if we use the positive cobordism from \autoref{sec::cobordisms}, we end up with Seifert fibered spheres that are known not to bound a negative definite 4--manifold. Similalry, if we use the positive definite cobordism $P$ from $-\Sigma(D^r(T_{p,q}))$ to $2S^3_{1/4}(D^{r-1}(T_{p,q})$ described in \cite{tesis}, we end up with 3--manifolds that Hedden \cite{hedden} shows do not bound a negative definite smooth 4--manifold since $\tau(T_{p,q})>0$. Also, the cobordisms constructed in \autoref{sec::cobordisms} do not keep track of the index of the iteration and so our technique cannot be used to prove independence of, for example, the family $\{D^r(T_{2,3})\}_{r\geq 1}$. Regardless, the following corollary provides a condition under which a family of iterated doubles of positive torus knots is independent.

\begin{cor}Let $\left\{\left(p_i,q_i\right)\right\}_i$ be a sequence of relatively prime positive integers, and $r_i$ a positive integer $(i=1,2,\ldots)$. If $p_{i}q_{i}\left(4p_{i}q_{i}-1\right)<p_{i+1}q_{i+1}\left(p_{i+1}q_{i+1}-1\right),$ then the family $\left\{D^{r_i}\left(T_{p_i,q_i}\right)\right\}_{i=1}^\infty$ is independent in $\mathcal{C}$.
\end{cor}

\bibliographystyle{abbrvnat}
\bibliography{References}

\begin{thebibliography}{26}
\providecommand{\natexlab}[1]{#1}
\providecommand{\url}[1]{\texttt{#1}}
\expandafter\ifx\csname urlstyle\endcsname\relax
  \providecommand{\doi}[1]{doi: #1}\else
  \providecommand{\doi}{doi: \begingroup \urlstyle{rm}\Url}\fi

\bibitem[Akbulut(1983)]{akbulut}
S.~Akbulut.
\newblock Observation.
\newblock CBMS conference at Santa Barbara, 1983.

\bibitem[Cochran and Gompf(1988)]{cochran-gompf}
T.~D. Cochran and R.~Gompf.
\newblock Applications of \text{D}onaldson's theorems to classical knot
  concordance, homology 3-spheres and \text{P}roperty \text{P}.
\newblock \emph{Topology}, 27\penalty0 (4):\penalty0 495--512, 1988.

\bibitem[Donaldson(1983{\natexlab{a}})]{d1}
S.~K. Donaldson.
\newblock An application of gauge theory to four dimensional topology.
\newblock \emph{Journal of Differential Geometry}, 18\penalty0 (2):\penalty0
  279--315, 1983{\natexlab{a}}.

\bibitem[Donaldson(1983{\natexlab{b}})]{d2}
S.~K. Donaldson.
\newblock Self-dual connections and the topology of smooth 4-manifolds.
\newblock \emph{Bulletin of the American Mathematical Society}, 8\penalty0
  (1):\penalty0 81--83, 1983{\natexlab{b}}.

\bibitem[Donaldson(1987)]{d3}
S.~K. Donaldson.
\newblock The orientation of \text{Y}ang-\text{M}ills moduli spaces and
  4-manifold topology.
\newblock \emph{J. Diff. Geo}, 26:\penalty0 397--428, 1987.

\bibitem[(Ed.) and Kirby(1995)]{kirby-problems}
R.~K. (Ed.) and E.~R. Kirby.
\newblock Problems in low-dimensional topology, 1995.

\bibitem[Fintushel and Stern(1990)]{FS-inst}
R.~Fintushel and R.~J. Stern.
\newblock Instanton homology of {S}eifert fibred homology three spheres.
\newblock \emph{Proceedings of the London Mathematical Society}, 3\penalty0
  (1):\penalty0 109--137, 1990.

\bibitem[Freedman(1982)]{freedman2}
M.~H. Freedman.
\newblock The topology of four-dimensional manifolds.
\newblock \emph{Journal of Differential Geometry}, 17\penalty0 (3):\penalty0
  357--453, 1982.

\bibitem[Freedman(1983)]{freedman}
M.~H. Freedman.
\newblock The disk theorem for four-dimensional manifolds.
\newblock In \emph{Proceedings of the International Congress of
  Mathematicians}, volume~1, page~2, 1983.

\bibitem[Freedman and Quinn(1990)]{freedman-quinn}
M.~H. Freedman and F.~Quinn.
\newblock \emph{Topology of 4-manifolds}, volume~39.
\newblock Princeton University Press Princeton, 1990.

\bibitem[Furuta(1990)]{furuta}
M.~Furuta.
\newblock Homology cobordism group of homology 3-spheres.
\newblock \emph{Inventiones Mathematicae}, 100\penalty0 (1):\penalty0 339--355,
  1990.

\bibitem[Hatcher(2002)]{Hatcher}
A.~Hatcher.
\newblock \emph{Algebraic topology}.
\newblock Cambridge University Press, Cambridge, 2002.
\newblock ISBN 0-521-79160-X; 0-521-79540-0.

\bibitem[Hedden(2007)]{hedden}
M.~Hedden.
\newblock Knot {F}loer homology of {W}hitehead doubles.
\newblock \emph{Geom. Topol.}, 11:\penalty0 2277--2338, 2007.
\newblock ISSN 1465-3060.
\newblock URL \url{https://doi.org/10.2140/gt.2007.11.2277}.

\bibitem[Hedden and Kirk(2012)]{hedden-kirk}
M.~Hedden and P.~Kirk.
\newblock Instantons, concordance, and {W}hitehead doubling.
\newblock \emph{Journal of Differential Geometry}, 91\penalty0 (2):\penalty0
  281--319, 2012.

\bibitem[Hoste(1986)]{hoste}
J.~Hoste.
\newblock A formula for {C}asson's invariant.
\newblock \emph{Trans. Amer. Math. Soc.}, 297\penalty0 (2):\penalty0 547--562,
  1986.
\newblock ISSN 0002-9947.
\newblock \doi{10.2307/2000539}.
\newblock URL \url{http://dx.doi.org/10.2307/2000539}.

\bibitem[Livingston and Melvin(1985)]{Livingston-Melvin}
C.~Livingston and P.~Melvin.
\newblock Abelian invariants of satellite knots.
\newblock In \emph{Geometry and topology ({C}ollege {P}ark, {M}d., 1983/84)},
  volume 1167 of \emph{Lecture Notes in Math.}, pages 217--227. Springer,
  Berlin, 1985.

\bibitem[Moser(1971)]{moser}
L.~Moser.
\newblock Elementary surgery along a torus knot.
\newblock \emph{Pacific Journal of Mathematics}, 38\penalty0 (3):\penalty0
  737--745, 1971.

\bibitem[Murasugi(2008)]{murasugi}
K.~Murasugi.
\newblock \emph{Knot theory \& its applications}.
\newblock Modern Birkh\"auser Classics. Birkh\"auser Boston, Inc., Boston, MA,
  2008.
\newblock ISBN 978-0-8176-4718-6.
\newblock \doi{10.1007/978-0-8176-4719-3}.
\newblock URL \url{http://dx.doi.org/10.1007/978-0-8176-4719-3}.
\newblock Translated from the 1993 Japanese original by Bohdan Kurpita, Reprint
  of the 1996 translation [MR1391727].

\bibitem[Neumann and Zagier(1985)]{nz}
W.~D. Neumann and D.~Zagier.
\newblock A note on an invariant of {F}intushel and {S}tern.
\newblock In \emph{Geometry and topology ({C}ollege {P}ark, {M}d., 1983/84)},
  volume 1167 of \emph{Lecture Notes in Math.}, pages 241--244. Springer,
  Berlin, 1985.

\bibitem[Park(2013)]{park}
K.~Park.
\newblock On independence of iterated whitehead doubles in the knot concordance
  group, 2013.

\bibitem[Pinz\'on-Caicedo(2017)]{tesis}
J.~Pinz\'on-Caicedo.
\newblock Independence of satellites of torus knots in the smooth concordance
  group.
\newblock \emph{Geom. Topol.}, 21\penalty0 (6):\penalty0 3191--3211, 2017.
\newblock ISSN 1465-3060.
\newblock URL \url{https://doi.org/10.2140/gt.2017.21.3191}.

\bibitem[Rudolph(1993)]{r1}
L.~Rudolph.
\newblock Quasipositivity as an obstruction to sliceness.
\newblock \emph{Bulletin of the American Mathematical Society}, 29\penalty0
  (1):\penalty0 51--59, 1993.

\bibitem[Rudolph(1995)]{r2}
L.~Rudolph.
\newblock An obstruction to sliceness via contact geometry and ``classical''
  gauge theory.
\newblock \emph{Inventiones mathematicae}, 119\penalty0 (1):\penalty0 155--163,
  1995.

\bibitem[Saveliev(2002)]{saveliev}
N.~Saveliev.
\newblock \emph{Invariants for homology {$3$}-spheres}, volume 140 of
  \emph{Encyclopaedia of Mathematical Sciences}.
\newblock Springer-Verlag, Berlin, 2002.
\newblock ISBN 3-540-43796-7.
\newblock \doi{10.1007/978-3-662-04705-7}.
\newblock URL \url{http://dx.doi.org/10.1007/978-3-662-04705-7}.
\newblock Low-Dimensional Topology, I.

\bibitem[Seifert(1950)]{Seifert}
H.~Seifert.
\newblock On the homology invariants of knots.
\newblock \emph{Quart. J. Math., Oxford Ser. (2)}, 1:\penalty0 23--32, 1950.

\bibitem[Thurston(1978)]{Thurston}
W.~Thurston.
\newblock Geometry and topology of three-manifolds.
\newblock 1978.
\newblock URL \url{http://library.msri.org/books/gt3m/}.

\end{thebibliography}
\end{document}